\documentclass[11pt,english,a4paper]{smfart}
\usepackage[english]{babel}
\usepackage{amssymb,xspace}
\usepackage{amstext}
\usepackage{smfthm}
\theoremstyle{plain}
\usepackage{amsbsy,amssymb,amsfonts,latexsym}

\usepackage{enumerate}
\usepackage{graphics}

\date{\today}

\title{A hypercyclic rank one perturbation of a unitary operator}

\author{Sophie Grivaux}
\address{
Laboratoire Paul Painlev\' e, UMR 8524, Universit\'e  Lille 1, Cit\' e Scientifique, 59655 Villeneuve d'Ascq
Cedex, France}
\email{grivaux@math.univ-lille1.fr}

\subjclass{47A16}

\keywords{Linear dynamical systems, hypercyclic operators, finite rank perturbations of unitary operators}
\thanks{This work was partially supported by ANR-Projet Blanc DYNOP}

\def\T{\ensuremath{\mathbb T}}
\def\R{\ensuremath{\mathbb R}}

\def\C{\ensuremath{\mathbb C}}

\def\N{\ensuremath{\mathbb N}}

\newcommand{\sep}{separable}

\newcommand{\hy}{hypercyclic}
\newcommand{\fhy}{frequently hypercyclic}
\newcommand{\ops}{operators}
\newcommand{\op}{operator}

\newcommand{\eve}{eigenvector}
\newcommand{\eva}{eigenvalue}
\newcommand{\ps}{perfectly spanning set of eigenvectors associated to unimodular
eigenvalues}

\newcommand{\mea}{measure}

\newcommand{\ifff}{if and only if}

\newcommand{\pss}[2]{\ensuremath{{\langle #1,#2\rangle}}}

\newtheorem{theorem}{Theorem}[section]

\newtheorem{lemma}[theorem]{Lemma}

{\theoremstyle{definition}}

{\theoremstyle{definition}}

{\theoremstyle{definition}\newtheorem{definition}[theorem]{Definition}}

{\theoremstyle{definition}}

{\theoremstyle{definition}}

\newtheorem{question}[theorem]{Question}

{\theoremstyle{definition}\newtheorem*{FFC Criterion}{Frequent
Faber-hypercyclicity Criterion}}
\newtheorem*{Hypercyclicity Criterion}{Hypercyclicity Criterion}
{\theoremstyle{definition}\newtheorem*{GS Criterion}{Godefroy-Shapiro
Criterion}}
\def\piednote#1{\let\oldfn=\thefootnote\def\thefootnote{}\footnote{\noindent#1}%
\addtocounter{footnote}{-1}\def\thefootnote{\oldfn}}

\begin{document}

\begin{abstract}
We prove that there exists a rank $1$ perturbation of a unitary operator on a complex separable infinite dimensional Hilbert space which is hypercyclic.
\end{abstract}

\maketitle

\section{Introduction}
We are interested in this note in the construction of some special \hy\ \ops\ on Hilbert spaces. Our work fits into the framework of linear dynamics, which is the study of the properties of the iterates $T^{n}$, $n\geq 0$, of a bounded linear \op\ $T$ acting on an infinite dimensional \sep\ Banach space $X$. It is of particular interest to study the behavior of the orbits $\mathcal{O}rb(x,T)=\{T^{n}x \textrm{ ; }n\geq 0\}$ of vectors $x$ of $X$ under the action of $T$. For instance when $\mathcal{O}rb(x,T)$ is dense in $X$, the vector $x$ is said to be \emph{\hy}. The \op\ $T$ itself is \hy\ when there exists an $x\in X$ such that $x$ is \hy\ for $T$. It is not completely trivial to exhibit \hy\ \ops: the first example of such an \op\ was given by Rolewicz \cite{R}, who proved that if $B$ denotes the standard backward weighted shift on $\ell_{2}(\N)$, $\lambda B$ is \hy\ for any complex number $\lambda $ with $|\lambda |>1$. Many more examples of \hy\ \ops\ on ``classical'' spaces can be found in the book \cite{BM}. It is a non-trivial result of Ansari \cite{A} and Bernal-Gonzalez \cite{Be}, relying on previous work of Salas \cite{S1} that any (real or complex) \sep\ infinite-dimensional Banach space $X$ supports a \hy\ \op. Such a general \op\ has necessarily the form $T=I+N$, where $N$ is a nuclear \op\ on $X$, so that a nuclear perturbation of the identity \op\ can indeed be \hy. Obviously a finite rank perturbation of the identity \op\ can never be \hy.
\par\smallskip
In \cite{Sh} Shkarin investigated the following question: can a finite rank perturbation of a unitary \op\ on a complex \sep\ infinite-dimensional Hilbert space be \hy? This question came from the work of Salas \cite{S2} on supercyclicity of weighted shifts: $T$ is said to be \emph{supercyclic} (a weaker requirement than hypercyclicity) if there exists a vector
$x\in X$ such that $\{\lambda T^{n}x \textrm{ ; }n\geq 0, \, \lambda \in \C\}$ is dense in $X$. It is known (\cite{Bo}, see also \cite{HW} and \cite{K}) that no hyponormal \op\ on a Hilbert space can be supercyclic. Salas thus proposed the following question: can a finite rank perturbation of a hyponormal \op\ on a Hilbert space be supercyclic? Shkarin answered in \cite{Sh} this question in the affirmative, and proved: there exists a unitary \op\ $V$ and an \op\ $R$ of rank at most $2$ acting on a Hilbert space $H$ such that $V+R$ is \hy\ on $H$. This yields an example of a contraction $A$ and a rank $1$ \op\ $S$ on $H$ such that $A+S$ is \hy. But  a natural question remained open in \cite{Sh}:

\begin{question}\label{q1}
 Does there exist a rank $1$ perturbation of a unitary \op\ on a Hilbert space which is \hy?
\end{question}

Our aim in this paper is to answer Question \ref{q1} in the affirmative:

\begin{theorem}\label{th1}
 There exists a unitary \op\ $U$ and a rank $1$ \op\ $R$ on the complex Hilbert space  $\ell_{2}(\N)$ such that the \op\ $T=U+R$ is \hy\ on $\ell_{2}(\N)$.
\end{theorem}

Our method of proof is rather different from the one employed in \cite{Sh}, the only common point being the criterion for hypercyclicity which we use: it is based on the properties of \eve s associated to \eva s of $T$ which are of modulus $1$, and was first introduced in \cite{BG}. We use here a recent refinement of this criterion which comes from \cite{G}, see Section 2 of this paper. The \ops\ which we construct are intrinsically different from the ones of \cite{Sh}: in \cite{Sh} the \ops\ live on the function space $L^{2}(\T)$, and the \op\ $V+R$ ($V$ unitary, $R$ of rank $2$) which is constructed is an \op\ induced by $V'+R'$ on an invariant subspace of 
$V'+R'$, where $V'$ is the multiplication \op\ by $z$ on $L^{2}(\T)$ and $R'$ is a rank one \op\ on $L^{2}(\T)$. One of the key tools in the proof of \cite{Sh} is a result of Belov \cite{Bel} concerning the distribution of values of certain functions $\varphi:\R\rightarrow\C$ defined as lacunary trigonometric series.
\par\smallskip
Our approach here is much more elementary: our unitary \op\ $U$ is a diagonal \op\ on $\ell_{2}(\N)$ with unimodular diagonal coefficients, and these coefficients as well as the two vectors $a$ and $b$ in $\ell_{2}(\N)$ which define $R=b\otimes a$ are constructed by induction in such a way that the \eve s associated to \eva s of modulus $1$ of the \op\ $U+R$ can be explicitly written down. The main idea of the proof of Theorem \ref{th1} is presented in Section 2, and the inductive construction, which is more technical, is given in Section 3.

\section{Main ingredients of the proof of Theorem \ref{th1}}

\subsection{A criterion for hypercyclicity}
The criterion for hypercyclicity which  we are going to use in the proof of Theorem \ref{th1} is stated in terms of \eve s associated to \eva s of modulus $1$ of the \op. Roughly speaking, if $T$ is a bounded linear \op\ on a complex \sep\ Banach space $X$ which has ``plenty'' of such \eve s, then $T$ is \hy. Here is the precise definition:

\begin{definition}
We say that $T\in \mathcal{B}(X)$ has a \emph{\ps}\ if there exists a continuous probability \mea\ $\sigma  $ on the unit circle $\T$ such that for any $\sigma  $-measurable subset $B$ of $\T$ with $\sigma  (B)=1$, we have $\overline{\textrm{sp}}[\ker(T-\lambda ) \textrm{ ; }\lambda \in B]=X$.
\end{definition}

When $T$ has a \ps, it is automatically \hy. This is proved in \cite{BG}. The easiest way to check this spanning property of the \eve s is to exhibit a family $(K_{i})_{i\geq 1}$ of compact perfect subsets of $\T$ and a family $(E_{i})_{i\geq 1}$ of \eve\ fields $E_{i}:K_{i}\rightarrow X$ which are continuous on $K_{i}$ and such that the vectors $E_{i}(\lambda )$, $i\geq 1,\lambda \in K_{i}$, span a dense subspace of $X$. This can be done rather easily by using the following theorem, which was proved in \cite{G}:

\begin{theorem}\label{th0}
Let $X$ be a complex \sep\ infinite-dimensional Banach space, and let $T$ be a bounded \op\ on $X$. Suppose that there exists a sequence $(u_{i})_{i\geq 1}$ of vectors of $X$ having the following properties:
 \begin{itemize}
  \item[(i)] for each $i\geq 1$, $u_{i}$ is an \eve\ of $T$ associated to an \eva\ $\mu _{i}$ of $T$, with $|\mu _{i}|=1$ and the $\mu_{i}$'s all distinct;
  
  \item[(ii)] $\textrm{sp}[u_{i}\textrm{ ; } i \geq 1]$ is dense in $X$;
  
  \item[(iii)] for any $i\geq 1$ and any $\varepsilon >0$, there exists an $n\not =i $ such that $||u_{n}-u_{i}||<\varepsilon $.
\end{itemize}
Then there exists a family $(K_{i})_{i\geq 1}$ of subsets of $\T$ which are homeomorphic to the Cantor set $2^{\omega }$ and a family $(E_{i})_{i\geq 1}$ of \eve\ fields $E_{i}:K_{i}\rightarrow X$ which are continuous on $K_{i}$ for each $i$ and which span $X$:
$\textrm{sp}[E_{i}(\lambda ) \textrm{ ; } i\geq 1,\,\lambda \in K_{i}]$ is dense in $X$. So $T$ has a \ps, and in particular $T$ is \hy.
\end{theorem}

Theorem \ref{th0} actually yields a stronger conclusion, as it is known that if $T$ has a \ps, then it is \fhy. See \cite{BG2} for the definition of frequent \hy ity and for a proof of this statement in the Hilbert space setting, and \cite{G} for a proof in the Banach space case. When the \eva s $\mu _{i}$ which appear in the assumption of Theorem \ref{th0} are $N^{th}$ roots of $1$, $N\geq 1$, then the \op\ is not only \hy\ but chaotic (it is \hy\ and has a dense set of periodic vectors). We will in the proof of Theorem \ref{th1} construct the \op\ $T$ so that the assumptions of Theorem \ref{th0} are satisfied. It will become clear in the course of the proof that we can choose the $\mu_{i}$'s to be $N^{th}$ roots of $1$, and thus the \op\ of Theorem \ref{th1} can be made chaotic and \fhy.

\subsection{Eigenvectors of rank one perturbations of diagonal operators}

We are looking for a \hy\ \op\ $T$ on the space $\ell_{2}(\N)$ endowed with the canonical basis $(e_{n})_{n\geq 1}$ of the form $T=U+R$, where $U$ is a unitary \op\ and $R$ is an \op\ of rank $1$. The unitary \op\ which we construct is a diagonal \op\ $D$ defined by $De_{n}=\lambda _{n}e_{n}$, $n\geq 1$, where $\lambda _{n}$ is for each $n\geq 1$  a complex number of modulus $1$ with the $\lambda _{n}$'s all distinct. The \op\ $R$ has the form $R=b\otimes a $, where $a=\sum_{n\geq 1}a_{n}e_{n}$ and $b=\sum_{n\geq 1}b_{n}e_{n}$ are two elements of $\ell_{2}(\N)$: $Rx=\pss{x}{b} a$ for any $x\in \ell_{2}(\N)$. Our aim is to define the coefficients $\lambda _{n}$ and the numbers $a_{n}$ and $b_{n}$ in such a way that the \op\ $T=D+R$ satisfies the assumptions of Theorem \ref{th0}.
\par\smallskip
Let $\lambda \in\T$ be a complex number of modulus $1$. Then with the notation above, $\lambda $ is an \eva\ of the \op\ $T=D+R$ with associated \eve\ $u\in\ell_{2}(\N)\setminus\{0\}$ \ifff\ $(D+R)u=\lambda u$, i.e. $Du+\pss{u}{b}a=\lambda u$, i.e. $(\lambda -D)u=\pss{u}{b}a$. If $\lambda \not\in\{\lambda _{n} \textrm{ ; } n\geq 1\}$, $\lambda -D$ is injective, and thus the equation above admits a non-zero solution $u$ if and only if $a\in \textrm{Ran}(\lambda -D)$, $a=(\lambda -D)a'$ where $a'\in\ell_{2}(\N)$ is unique and 
$\pss{a'}{b}=1$. If $a=\sum_{n\geq 1} a_{n}e_{n}$, then necessarily $$a'=\sum_{n\geq 1}\frac{a_{n}}{\lambda -\lambda _{n}}e_{n},$$ and $\pss{a'}{b}=1$ means that $$\sum_{n\geq 1}\frac{a_{n}\overline{b}_{n}}{\lambda -\lambda _{n}}=1.$$
We can reformulate this observation as:

\begin{lemma}\label{lem1}
 If $\lambda\in\T\setminus\{\lambda_n \textrm{ ; } n\geq 1\}$, then $\lambda$ is an \eva\ of $D+R$ \ifff\ 
$$ \sum_{n\geq 1} \left|\frac{a_n}{\lambda-\lambda_n}\right|^2<+\infty \quad \textrm{ and } \quad 
\sum_{n\geq 1} \frac{a_n\overline{b}_n}{\lambda-\lambda_n}=1.$$ In this case an associated \eve\ $u$ is given by 
$$u=\sum_{n\geq 1} \frac{a_n}{\lambda-\lambda_n}e_n.$$
\end{lemma}

\subsection{Strategy of the proof of Theorem \ref{th1}}
Let $j:\{1,2,\ldots\}\longrightarrow \{1,2,\ldots\}$ be a function having the following properties:

$\bullet$ $j(1)=1$;

$\bullet$ $j(n)<n$ for every $n\geq 2$;

$\bullet$ for any $k\geq 1$ the set $\{n\geq 2 \textrm{ ; } j(n)=k\}$ is infinite, i.e. $j$ takes every value $k\geq 1$ infinitely often.
\par\smallskip
The proof of Theorem \ref{th1} will be carried out via an induction argument. As Step $n$, $n\geq 1$, we define two unimodular numbers $\lambda _{n}$ and $\mu _{n}$,
a complex number $a_{n}$ and
an $n$-tuple $b^{(n)}=(b_{1}^{(n)},\ldots, b_{n}^{(n)})$ of complex numbers
such that the following properties hold true:

\begin{itemize}
 \item[(1)] if $D_{n}$ denotes the diagonal \op\ on $\C^{n}$ with diagonal coefficients $\lambda _{1},\ldots, \lambda _{n}$, with
$\lambda_n\not\in\{\lambda_1,\ldots,\lambda_{n-1}\}$,
 and $R_{n}$ denotes the rank $1$ \op\ $b^{(n)}\otimes a^{(n)}$ on $\C^{n}$, i.e. 
 $R_{n}x=\pss{x}{b^{(n)}}a^{(n)}$ for any $x\in\C^{n}$, where $a^{(n)}=\sum_{j=1}^{n}a_{j}e_{j}$ and $b^{(n)}=\sum_{j=1}^{n}b_{j}^{(n)}e_{j}$, then the \op\ $T_{n}=D_{n}+R_{n}$ acting on 
 $\C^{n}$ has $n$ distinct \eva s which are the unimodular numbers $\mu _{1},\ldots, \mu _{n}$. Moreover
 $\mu _{n}$ does not belong to the set of distinct numbers $\{\lambda _{1},\ldots, \lambda _{n},\mu_{1},\ldots, \mu _{n-1}\}$, and the vector
 $$u_{i}^{(n)}=\sum_{j=1}^{n}\frac{a_{j}}{\mu _{i}-\lambda _{j}}e_{j}$$ is an \eve\ of $T_{n}$ associated to the \eva\ $\mu _{i}$. Additionally for any $n\geq 1$, $\textrm{sp}[u_{i}^{(n)} \textrm{ ; } i=1,\ldots,n]=\textrm{sp}[e_{1},\ldots, e_{n}]$. Thus there exists a positive constant $C_{n}$ such that for any $x\in\C^{n}$ with $x=\sum_{j=1}^{n}x_{j}e_{j}=\sum_{i=1}^{n}\alpha _{i}u_{i}^{(n)}$, we have
 $$\sum_{i=1}^{n}|\alpha _{i}|\leq C_{n}\left(\sum_{j=1}^{n}|x_{j}|^{2}\right)^{\frac{1}{2}}$$
 
 \item[(2)] $C_{n}>C_{n-1}$ and $C_n>2$
 
 \item[(3)]  $0<|a_{n}|<2^{-n}$
 
 \item[(4)] $|b_{n}^{(n)}|<2^{-n}$
 
 \item[(5)] we have $$\left(\sum_{i=1}^{n-1}|b_{i}^{(n)}-b_{i}^{(n-1)}|^{2}\right)^{\frac{1}{2}}<2^{-n}$$
 
 \item[(6)] for any $i=1,\ldots,n-1$, $$||u_{i}^{(n)}-u_{i}^{(n-1)}||<\frac{2^{-n}}{C_{n-1}}$$
 
 \item[(7)] $||u_{j(n)}^{(n)}-u_{n}^{(n)}||<2^{-n}$
 
 \item[(8)] for any $k=1,\ldots, n-1$ and any $i=1,\ldots,k$, $||T_{n}u_{i}^{(k)}-\mu_{i}u_{i}^{(k)}||<3\,.\,2^{-(k-1)}$.
\end{itemize}
\par\smallskip
Suppose that the construction of the sequences $(\lambda _{n})_{n\geq 1}$, $(\mu _{n})_{n\geq 1}$, $(a_{n})_{n\geq 1}$ and $(b^{(n)})_{n\geq 1}$
has been carried out in such a way that properties (1)-(8) are satisfied. 
By (2) the vector $a=\sum_{n\geq 1}a_{n}e_{n}$ belongs to $\ell_{2}(\N)$. By (4) and (5), we have 
\par\smallskip
\begin{itemize}
 \item[(9')] $||b^{(n)}-b^{(n-1)}||=||\sum_{i=1}^{n-1}(b_{i}^{(n)}-b_{i}^{(n-1)})e_{i}+b_{n}^{(n)}e_{n}||< 2.2^{-n}=2^{-(n-1)}$
\end{itemize}
\par\smallskip
so that the sequence $(b^{(n)})_{n\geq 1}$ converges in $\ell_{2}(\N)$ to a certain vector $b=\sum_{n\geq 1}b_{n}e_{n}$, with
\begin{itemize}
\par\smallskip
 \item[(10')] $||b^{(n)}-b||\leq\sum_{j\geq n}||b^{(j+1)}-b^{(j)}||\leq \sum_{j\geq n} 2^{-j} < 2^{-(n-1)}$.
\end{itemize}
\par\smallskip
 So it makes sense to define the rank one \op\ $R=b\otimes a$ on $\ell_{2}(\N)$. Let $D$ be the diagonal \op\ $D=\textrm{diag}(\lambda _{n}\textrm{ ; }n\geq 1)$ on $\ell_{2}(\N)$. We are going to show, using Theorem \ref{th0}, that $D+R$ is then \hy, which will prove Theorem \ref{th1}.
\par\smallskip
\begin{proof}[Proof of Theorem \ref{th1} modulo the inductive construction]
For any $n\geq 1$, let $P_{n}$ denote the canonical projection of $\ell_{2}(\N) $ onto $\textrm{sp}[e_{1},\ldots, e_{n}]$. For any $x=\sum_{j\geq 1}x_{j}e_{j}\in \ell_{2}(\N)$, we have
$$T_{n}P_{n} x=T_{n}\left(\sum_{j= 1}^{n}x_{j}e_{j}\right)= \sum_{j= 1}^{n}\lambda _{j}x_{j}e_{j}+
\pss{x}{b^{(n)}}a^{(n)}.$$ 
Since $a^{(n)}\rightarrow a$, $b^{(n)}\rightarrow b$ and $\sup_{n\geq 1}||b^{(n)}||$ is finite, 
%
%
$||T_{n}P_{n}x-(D+R)x||$ tends to zero as $n$ tends to infinity. Applying this to $x=u_{i}^{(k)}$ yields that for any $k\geq 1$ and any $i=1,\ldots,k$,
$||Tu_{i}^{(k)}-\mu _{i}u_{i}^{(k)}||\leq 3\,.\, 2^{-(k-1)}$ by (8), as $T_{n}P_{n}u_{i}^{(k)}=T_{n}u_{i}^{(k)}$ for any $n\geq k$. By (6) the sequence $(u_{i}^{(n)})_{n\geq i}$ converges as $n$ tends to infinity to a certain vector $u_{i}\in\ell_{2}(\N)$, which is nothing but
$$u_{i}=\sum_{j=1}^{+\infty }\frac{a_{j}}{\mu _{i}-\lambda _{j}}e_{j}.$$ It is a non zero vector, and making $k$ tend to infinity in the inequalities above shows that $Tu_{i}=\mu _{i}u_{i}$, so that $u_{i}$ is an \eve\ of $T$ associated to the \eva\ $\mu _{i}$. 
\par\smallskip
Let us now prove that the sequence $(u_{i})_{i\geq 1}$ satisfies the assumptions of Theorem \ref{th0}: assertion ${(i)}$ is true by construction, as the $\mu _{i}$'s are all distinct. As for assertion ${(ii)}$, let us consider a vector $x=\sum_{j=1}^{r}x_{j}e_{j}$ with finite support and $||x||\leq 1$. Writing $x$ as $x=\sum_{i=1}^{r}\alpha _{i}u_{i}^{(r)}$, we have by (1)
\begin{eqnarray*}
||x-\sum_{i=1}^{r}\alpha _{i}u_{i}||&\leq& \left(\sum_{i=1}^{r}|\alpha _{i}|\right)\, \sup_{i=1,\ldots, r}||u_{i}-u_{i}^{(r)}||\\
&\leq& C_{r}\, ||x||\,\sup_{i=1,\ldots,r}\sum_{k\geq r+1}||u_{i}^{(k)}-u_{i}^{(k-1)}||\\
&\leq& C_{r} \, \sum_{k\geq r+1}\frac{2^{-k}}{C_{k-1}}\leq 2^{-r}
\end{eqnarray*}
by (6).
Hence for any $\varepsilon >0$ there exists a vector $y\in\textrm{sp}[u_{j}\textrm{ ; }j\geq 1]$ such that $||x-y||<\varepsilon $, and this proves assertion $(ii)$. Assertion $(iii)$ is a consequence of (7): for any $k\geq 1$ let $A_{k}$ be the set $A_{k}=\{n\geq 2 \textrm{ ; } j(n)=k\}$.  Observe that if $n\in A_k$, $n\geq k+1$. For any $n\in A_{k}$ we have $||u_{k}^{(n)}-u_{n}^{(n)}||<2^{-n}$ by (7). Let us estimate $||u_{n}-u_{k}||$:
\begin{eqnarray*}
 ||u_{n}-u_{k}||&\leq& ||u_{n}-u_{n}^{(n)}||+||u_{n}^{(n)}-u_{k}^{(n)}||+||u_{k}^{(n)}-u_{k}||\\
 &\leq&\sum_{m\geq n+1}||u_{n}^{(m)}-u_{n}^{(m-1)}||+2^{-n}+\sum_{m\geq n+1}||u_{k}^{(m)}-u_{k}^{(m-1)}||\\
 &\leq& 2\, \sum_{m\geq n+1}2^{-m}+2^{-n}=5.2^{-n}.
\end{eqnarray*}
Thus  if $\varepsilon $ is any positive number, since $A_{k}$ is infinite there exists an $n\in A_{k}$ such that $||u_{n}-u_{k}||<\varepsilon $, and assertion $(iii)$ of Theorem \ref{th0} is satisfied too. We have thus proved that $T$ is \hy, which proves Theorem \ref{th1} modulo the construction of $\lambda _{n}$, $\mu _{n}$, $a_{n}$ and $b^{(n)}$ for each $n\geq 1$.
\end{proof}

\section{The induction step}
In order to complete the proof of Theorem \ref{th1}, we now have to carry out the induction step. Before starting, let us reformulate the first half of condition (1) in a more convenient way: saying that the \op\ $T_n=D_n+R_n$ acting on $\C^{n}$ has $n$ distinct \eva s $\mu _{1},\ldots, \mu _{n}$ exactly means that we have
\begin{equation}\label{eq1}\tag{E}
\sum_{j=1} ^{n}\frac{a_{j}\overline{b}_{j}^{(n)}}{\mu _{i}-\lambda _{j}}=1
\qquad \textrm{for any }  i=1,\ldots,n.
\end{equation}

Let $M_{n}\in\mathcal{M}_{n}(\C)$ be the matrix $M_{n}=(m_{ij})_{1\leq i,j\leq n}$ with $m_{ij}=\frac{1}{\mu _{i}-\lambda _{j}}$. These coefficients are well-defined, as we choose at each step $k$
$\lambda_{k}\not\in\{\mu_{1},\ldots, \mu _{k}\}$
and $\mu _{k}\not\in\{\lambda _{1},\ldots, \lambda _{k}\}$. Then equations (\ref{eq1}) can be rewritten as the matrix equation
$$
\begin{pmatrix}
 \dfrac{1}{\mu _{1}-\lambda _{1}}&\ldots&\dfrac{1}{\mu _{1}-\lambda _{n}}\\
 \vdots&&\vdots\\
  \dfrac{1}{\mu _{n}-\lambda _{1}}&\ldots&\dfrac{1}{\mu _{n}-\lambda _{n}}
\end{pmatrix}
.
\begin{pmatrix}
a_{1} \overline{b}_{1}^{(n)}\\
a_{2} \overline{b}_{2}^{(n)}\\
\vdots\\
a_{n} \overline{b}_{n}^{(n)}
\end{pmatrix}
=
\begin{pmatrix}
1\\
1\\
\vdots\\
1
\end{pmatrix},
\quad 
\textrm{ i.e. }
\quad 
M_{n}
\begin{pmatrix}
a_{1} \overline{b}_{1}^{(n)}\\
a_{2} \overline{b}_{2}^{(n)}\\
\vdots\\
a_{n} \overline{b}_{n}^{(n)}
\end{pmatrix}
=
\begin{pmatrix}
1\\
1\\
\vdots\\
1
\end{pmatrix}.
$$
\par\smallskip
We are now ready to begin the construction.
\par\smallskip
$\bullet$ We start by taking $\lambda _{1}=1$ and $a_{1}=4^{-1}$ for instance. Then we take $\mu _{1}\in\T$ with $\mu _{1}\not=\lambda _{1}$ and $|\mu _{1}-\lambda _{1}|$ so small (with $|\mu _{1}-\lambda _{1}|<1$ in particular) that if we set $$\overline{b}_{1}^{(1)}=\frac{\mu _{1}-\lambda _{1}}{a_{1}},$$ then $|b_{1}^{(1)}|<2^{-1}$. Of course $Te_{1}=\mu _{1}e_{1}$.
\par\smallskip
$\bullet$ Suppose now that the construction has been carried out until Step $n-1$. We have to construct $\lambda _{n}\in\T$, $\mu _{n}\in\T$, $a_{n}\in\C$ and $b^{(n)}\in\C^{n}$ such that properties (1)-(8) hold true. First of all, let $\varepsilon >0$ be a positive number which is so small that:

\begin{align}
\tag{a} & 0<\varepsilon <4^{-(n+1)}\\
\tag{b} &\prod_{j=1}^{n}(1+2^{-j})\left(\sum_{j=1}^{n-1}\frac{1}{|\mu _{j(n)}-\lambda _{j}|^{2}}\right)^{\frac{1}{2}} \varepsilon <4^{-(n+1)}\\
\tag{c} &\frac{1}{\min_{j=1,\ldots, n-1}|a_{j}|}\left(1+\left(\sum_{j=1}^{n-1}|a_{j}|^{2}\right)^{\frac{1}{2}}\right)
 \prod_{j=1}^{n}(1+2^{-j})\,\varepsilon <2^{-n}.
\end{align}
We first construct the $n^{th}$ diagonal coefficient $\lambda _{n}$ of $D_{n}$: it is chosen very close to $\mu _{j(n)}$. More precisely: by the induction assumption $\mu _{j(n)}$ is an \eva\ of the matrix $M_{n-1}$, so that
\begin{equation*}
\sum_{j=1} ^{n-1}\frac{a_{j}\overline{b}_{j}^{(n-1)}}{\mu _{j(n)}-\lambda _{j}}=1.
\end{equation*}
It follows that there exists $\delta >0$ such that for any $\lambda \in\T\setminus\{\lambda _{1},\ldots, \lambda _{n-1}\}$ with $|\lambda -\mu _{j(n)}|<\delta $, we have
\begin{align*}
 & \bullet \quad \left|1-\sum_{j=1} ^{n-1}\frac{a_{j}\overline{b}_{j}^{(n-1)}}{\lambda -\lambda _{j}} \right|<\varepsilon \\
 & \bullet \quad \prod_{j=1}^{n}(1+2^{-j})\left(\sum_{j=1}^{n-1}\frac{1}{|\lambda -\lambda _{j}|^{2}}\right)^{\frac{1}{2}} \varepsilon <4^{-(n+1)}\\
 & \bullet \quad \left(\sum_{j=1}^{n-1}|a_{j}|^{2}\,.\,\left|\frac{1}{\mu _{j(n)}-\lambda _{j}}-\frac{1}{\lambda -\lambda _{j}}\right|^{2}\right)^{\frac{1}{2}}<\varepsilon .
\end{align*}
We choose $\lambda _{n}\in\T\setminus\{\lambda _{1},\ldots,\lambda _{n-1},\mu _{1},\ldots,\mu _{n-1}\}$ such that $|\lambda _{n}-\mu _{j(n)}|<\delta $. We then have:
\begin{align}
 & \tag{d} \left|1-\sum_{j=1} ^{n-1}\frac{a_{j}\overline{b}_{j}^{(n-1)}}{\lambda_{n} -\lambda _{j}} \right|<\varepsilon \\
 & \tag{e} \prod_{j=1}^{n}(1+2^{-j})\left(\sum_{j=1}^{n-1}\frac{1}{|\lambda_{n} -\lambda _{j}|^{2}}\right)^{\frac{1}{2}} \varepsilon <4^{-(n+1)}\\
 & \tag{f} \left(\sum_{j=1}^{n-1}|a_{j}|^{2}\,.\,\left|\frac{1}{\mu _{j(n)}-\lambda _{j}}-\frac{1}{\lambda_{n} -\lambda _{j}}\right|^{2}\right)^{\frac{1}{2}}<\varepsilon .
\end{align}
Once $\lambda _{n}$ is chosen, the next step is to choose $\mu _{n}$. We take $\mu _{n}\in\T\setminus\{\lambda _{1},\ldots,\lambda _{n},\mu _{1},\ldots,\mu _{n-1}\}$ with $|\mu _{n}-\lambda _{n}|$
so small that
\begin{align}
 & \tag{g} \left|1-\sum_{j=1} ^{n-1}\frac{a_{j}\overline{b}_{j}^{(n-1)}}{\mu _{n} -\lambda _{j}} \right|<\varepsilon \\
 & \tag{h} \prod_{j=1}^{n}(1+2^{-j})\left(\sum_{j=1}^{n-1}\frac{1}{|\mu _{n} -\lambda _{j}|^{2}}\right)^{\frac{1}{2}} \varepsilon <4^{-(n+1)}\\
 & \tag{i} \left(\sum_{j=1}^{n-1}|a_{j}|^{2}\,.\,\left|\frac{1}{\mu _{j(n)}-\lambda _{j}}-\frac{1}{\mu _{n} -\lambda _{j}}\right|^{2}\right)^{\frac{1}{2}}<\varepsilon 
 \end{align}
 and
 \begin{align}
 & \tag{j} \frac{|\mu _{n}-\lambda _{n}|}{|\mu _{i}-\lambda _{n}|}<\frac{2^{-n}}{C_{n-1}}\quad \textrm{for any } i=1,\ldots,n-1\\
 & \tag{k}||M_{n}^{-1}||\leq(1+2^{-n})\,||M_{n-1}^{-1}||.
\end{align}

It is easy to see that conditions (g), (h), (i) and (j) can be fullfilled if $|\mu _{n}-\lambda _{n}|$ is small enough. That condition (k) can be made to hold too is not so immediate, but not too hard either: first of all for any $\varepsilon '>0$ there exists a $\delta '>0$ such that if $|\mu _{n}-\lambda _{n}|<\delta '$, then
$$\left|\frac{\det M_{n-1}}{(\mu _{n}-\lambda _{n})\det M_{n}}-1\right|<\varepsilon '.$$ 
Indeed 
$(\mu _{n}-\lambda _{n})\det M_{n}=\det \tilde{M}_{n}$, where $ \tilde{M}_{n}$ is the matrix obtained from $M_{n}$ by multiplying its last line by $(\mu _{n}-\lambda _{n})$. If $|\mu _{n}-\lambda _{n}|$
is extremely small, the coefficients $( \tilde{M}_{n})_{nj}$, $j=1,\ldots, n-1$, are almost equal to zero, while
$( \tilde{M}_{n})_{nn}=1$. Thus $\det  \tilde{M}_{n}$ can be made as close as we wish to $\det M_{n-1}$, and it is possible to ensure that $$\left|\frac{1}{(\mu _{n}-\lambda _{n})\det M_{n}}-\frac{1}{\det M_{n-1}}\right|<
\frac{\varepsilon '}{|\det M_{n-1}|},$$
from which it follows that
$$\left|\frac{\det M_{n-1}}{(\mu _{n}-\lambda _{n})\det M_{n}}-1\right|<\varepsilon '.$$ 
Notice that $$ \left|\frac{1}{\det M_{n}}-\frac{\mu _{n}-\lambda _{n}}{\det M_{n-1}}\right|<\varepsilon ' \frac{|\mu _{n}-\lambda _{n}|}{|\det M_{n-1}|} \cdot $$
%
%
\par\smallskip
Then the formula
$M_{n}^{-1}=\frac{1}{\det M_{n}}{}^{t}\textrm{com}M_{n}$ yields that:

-- the coefficients $(n,j)$ and $(i,n)$ of $M_{n}^{-1}$, $i,j=1,\ldots,n$, can be made arbitrarily small if $|\mu _{n}-\lambda _{n}|$ is small enough, as
$({}^{t}\textrm{com}M_{n})_{nj} $ and $({}^{t}\textrm{com}M_{n})_{in}$ do not depend on $|\mu_n-\lambda_n|$, while $\det M_n$ can be made arbitrarily small with $|\mu _{n}-\lambda _{n}|$;

-- the coefficients $(i,j)$, $i,j=1,\ldots,n-1$ can be made very close to the coefficients $(M_{n-1}^{-1})_{ij}$. Indeed the dominant term in the computation of $({}^{t}\textrm{com}M_{n})_{ij}$ is the one involving $\frac{1}{\mu _{n}-\lambda _{n}}$, that is $\frac{1}{\mu _{n}-\lambda _{n}} ({}^{t}\textrm{com}M_{n-1})_{ij}$. So $(M_{n}^{-1})_{ij}$ can be made as close as we wish to $$\frac{1}{(\mu _{n}-\lambda _{n})\det M_{n}}({}^{t}\textrm{com}M_{n-1})_{ij}=
\frac{\det M_{n-1}}{(\mu _{n}-\lambda _{n})\det M_{n}}(M_{n-1}^{-1})_{ij}.$$

Hence $M_{n}^{-1}$ is very close to the matrix $A_{n}$ for the \op\ norm on $\mathcal{M}_{n}(\C)$, where $(A_{n})_{ij}=(M_{n-1}^{-1})_{ij}$ for $i,j=1,\ldots,n-1$ and $(A_{n})_{in}=(A_{n})_{nj}=0$ for $i,j=1,\ldots,n$. Hence there exists $\gamma >0$ such that $||M_{n}^{-1}||\leq (1+2^{-n})||M_{n-1}^{-1}||$ if $|\mu _{n}-\lambda _{n}|<\gamma $, and property (k) is satisfied if $\mu _{n}$ is sufficiently close to $\lambda _{n}$.
\par\smallskip
Now that $\lambda _{n}$ and $\mu _{n}$ are constructed, it remains to fix $a_{n}$ and $b^{(n)}$. We take first $$a_{n}=2^{-(n+1)}|\mu _{n}-\lambda _{n}|.$$ There is now not much room for the choice of $b^{(n)}$: we must have
$$M_{n}\,
\begin{pmatrix}
 a_{1}\overline{b}_{1}^{(n)}\\
 \vdots\\
 a_{n}\overline{b}_{n}^{(n)}
\end{pmatrix}=
\begin{pmatrix}
 1\\
 \vdots\\
 1
\end{pmatrix}
\quad \textrm{i.e.}\quad 
\begin{pmatrix}
 a_{1}\overline{b}_{1}^{(n)}\\
 \vdots\\
 a_{n}\overline{b}_{n}^{(n)}
\end{pmatrix}= M_{n}^{-1}
\begin{pmatrix}
 1\\
 \vdots\\
 1
\end{pmatrix}.
$$
The numbers $\overline{b}_{j}^{(n)}$ are completely determined by these equations, and so we set
$$\overline{b}_{i}^{(n)}=\frac{1}{a_{i}}\sum_{j=1}^{n}(M_{n}^{-1})_{ij}.$$
\par\smallskip
It now remains to check that with this construction, properties (1)-(8) are satisfied:
\par\smallskip
$\bullet$ property (1) is true by construction, since
$$M_{n}\,
\begin{pmatrix}
 a_{1}\overline{b}_{1}^{(n)}\\
 \vdots\\
 a_{n}\overline{b}_{n}^{(n)}
\end{pmatrix}=
\begin{pmatrix}
 1\\
 \vdots\\
 1
\end{pmatrix}.$$ 
\par\smallskip
$\bullet$ property (2) is trivially true if $C_{n}$ is sufficiently large.
\par\smallskip
$\bullet$ as $a_{n}=2^{-(n+1)}|\mu _{n}-\lambda _{n}|$, $0<|a_{n}|<2^{-n}$, so (3) is true.
\par\smallskip
$\bullet$ let us now check property (5). We have
$$M_{n-1}\,
\begin{pmatrix}
 a_{1}\overline{b}_{1}^{(n-1)}\\
 \vdots\\
 a_{n-1}\overline{b}_{n-1}^{(n-1)}
\end{pmatrix}=
\begin{pmatrix}
 1\\
 \vdots\\
 1
\end{pmatrix}.
$$
Hence
$$
M_{n}\,
\begin{pmatrix}
 a_{1}\overline{b}_{1}^{(n-1)}\\
 \vdots\\
 a_{n-1}\overline{b}_{n-1}^{(n-1)}\\
 0
\end{pmatrix}=
\begin{pmatrix}
 1\\
 \vdots\\
 1\\
 c_{n}
\end{pmatrix}
\quad \textrm{where} \quad c_{n}=\sum_{j=1}^{n-1}\frac{a_{j}\overline{b}_{j}^{(n-1)}}{\mu _{n}-\lambda _{j}}\cdot$$
By (g) we have $|1-c_{n}|<\varepsilon $, so that
$$\left|\left|
M_{n}\,
\begin{pmatrix}
 a_{1}(\overline{b}_{1}^{(n)}-\overline{b}_{1}^{(n-1)})\\
 \vdots\\
 a_{n-1}(\overline{b}_{n-1}^{(n)}- \overline{b}_{n-1}^{(n-1)})\\
  a_{n}\overline{b}_{n}^{(n)}
\end{pmatrix}
\right|\right|=|1-c_{n}|<\varepsilon .$$
Hence
$$\left|\left|
\begin{pmatrix}
 a_{1}(\overline{b}_{1}^{(n)}- \overline{b}_{1}^{(n-1)})\\
 \vdots\\
 a_{n-1}(\overline{b}_{n-1}^{(n)}- \overline{b}_{n-1}^{(n-1)})\\
  a_{n}\overline{b}_{n}^{(n)}
\end{pmatrix}
\right|\right|< \varepsilon \,||M_{n}^{-1}||\leq \varepsilon \,(1+2^{-n})\,||M_{n-1}^{-1}||\leq\ldots\leq
\varepsilon \prod_{j=1}^{n}(1+2^{-j})
$$
by (k) and the fact that $||M_1^{-1}||=|\mu_1-\lambda_1|<1$,
that is
$$\left(\sum_{j=1}^{n-1}|a_{j}|^{2}\,|b_{j}^{(n)}-b_{j}^{(n-1)}|^{2}+|a_{n}b_{n}^{(n)}|^{2}\right)^{\frac{1}{2}}<\varepsilon \prod_{j=1}^{n}(1+2^{-j}).$$
In particular
\begin{align}
 \tag{l} \left(\sum_{j=1}^{n-1}|a_{j}|^{2}\,|b_{j}^{(n)}-b_{j}^{(n-1)}|^{2}\right)^{\frac{1}{2}}<\varepsilon \prod_{j=1}^{n}(1+2^{-j})
\end{align}
 so that
 $$\min_{j=1,\ldots,n-1}\,|a_{j}|\,\left(\sum_{j=1}^{n-1}|b_{j}^{(n)}-b_{j}^{(n-1)}|^{2}\right)^{\frac{1}{2}}<\varepsilon \prod_{j=1}^{n}(1+2^{-j}).$$ By (c) we get that
 $$\left(\sum_{j=1}^{n-1}|b_{j}^{(n)}-b_{j}^{(n-1)}|^{2}\right)^{\frac{1}{2}}<2^{-n},$$ which is property (5).
 \par\smallskip
$\bullet$ property (4) is a consequence of the equations
$$\sum_{j=1} ^{n}\frac{a_{j}\overline{b}_{j}^{(n)}}{\mu _{n} -\lambda _{j}} =1,\quad \textrm{i.e.}\quad 
\sum_{j=1} ^{n-1}\frac{a_{j}\overline{b}_{j}^{(n)}}{\mu _{n} -\lambda _{j}} +\frac{a_{n}\overline{b}_{n}^{(n)}}{\mu _{n} -\lambda _{n}} =1$$
and
$$\sum_{j=1} ^{n-1}\frac{a_{j}\overline{b}_{j}^{(n-1)}}{\mu _{j(n)} -\lambda _{j}} =1.$$
We have
\begin{eqnarray*}
a_{n}\overline{b}_{n}^{(n)}&=&(\mu _{n}-\lambda _{n})\left(1-
\sum_{j=1} ^{n-1}\frac{a_{j}\overline{b}_{j}^{(n)}}{\mu _{n} -\lambda _{j}} \right)
=(\mu _{n}-\lambda _{n})\left(\sum_{j=1}^{n-1}\left(
\frac{a_{j}\overline{b}_{j}^{(n-1)}}{\mu _{j(n)} -\lambda _{j}} -\frac{a_{j}\overline{b}_{j}^{(n)}}{\mu _{n} -\lambda _{j}} 
\right)\right)\\
&=& (\mu _{n}-\lambda _{n})\left(
\sum_{j=1}^{n-1} a_{j} \left(\frac{1}{\mu _{j(n)} -\lambda _{j}}-\frac{1}{\mu _{n} -\lambda _{j}}\right)\overline{b}_{j}^{(n-1)}\right.\\
&+&\left.\sum_{j=1}^{n-1} \frac{a_{j}}{\mu _{n}-\lambda _{j}}(\overline{b}_{j}^{(n-1)}-\overline{b}_{j}^{(n)})
\right).
\end{eqnarray*}
Thus
\begin{eqnarray*}
 |a_{n}\overline{b}_{n}^{(n)}|&\leq&|\mu _{n}-\lambda _{n}|\left(
 \sum_{j=1}^{n-1}|a_{j}|^{2}\left|\frac{1}{\mu _{j(n)}-\lambda _{j}}-\frac{1}{\mu _{n}-\lambda _{j}}\right|^{2}
 \right)^{\frac{1}{2}}\,\left(\sum_{j=1}^{n-1}|{b}_{j}^{(n-1)}|^{2}\right)^{\frac{1}{2}}\\
 &+&|\mu _{n}-\lambda _{n}|\left(\sum_{j=1}^{n-1}|a_{j}|^{2}|{b}_{j}^{(n-1)}-{b}_{j}^{(n)}|^{2}\right)^{\frac{1}{2}}\left(\sum_{j=1}^{n-1}\frac{1}{|\mu _{n}-\lambda _{j}|^{2}}\right)^{\frac{1}{2}}.
 \end{eqnarray*}
 Now by (i) and (l), we have
 \begin{eqnarray*}
|a_{n}\overline{b}_{n}^{(n)}| &\leq&|\mu _{n}-\lambda _{n}|\left(\varepsilon \, ||b^{(n-1)}||
 +\varepsilon \,\prod_{j=1}^{n}(1+2^{-j})\left(\sum_{j=1}^{n-1}\frac{1}{|\mu _{n}-\lambda _{j}|^{2}}\right)^{\frac{1}{2}}
 \right).
\end{eqnarray*}
We have seen in Section 2.3 that properties (4) and (5) at Step $j\leq n-1$ imply that $||b^{(j)}-b^{(j-1)}||\leq 2^{-(j-1)}$ (this is assertion (9')), so that $||b^{(n-1)}||\leq\sum_{j=2}^{n-1}2^{-(j-1)}\leq 1$. Combining this with property (h), we obtain that
$$
|a_{n}\overline{b}_{n}^{(n)}| <|\mu _{n}-\lambda _{n}|(\varepsilon +4^{-(n+1)}).
$$
Since $\varepsilon <4^{-(n+1)}$ by (a), 
$$
|{b}_{n}^{(n)}| <2\,.\,4^{-(n+1)}\frac{|\mu _{n}-\lambda _{n}|}{|a_{n}|}\cdot
$$
As
$a_{n}=2^{-(n+1)}|\mu _{n}-\lambda _{n}|$, this yields that $|b_{n}^{(n)}|<2^{-n}$, and (4) holds true.
\par\smallskip
$\bullet$ property (6) is easy: for $i=1,\ldots, n-1$,
\begin{eqnarray*}
 ||u_{i}^{(n-1)}-u_{i}^{(n)}||&=&
\frac{|a_{n}|}{|\mu _{i}-\lambda _{n}|}=2^{-(n+1)}\frac{|\mu _{n}-\lambda _{n}|}{|\mu _{i}-\lambda _{n}|}<{2^{-n}}
\end{eqnarray*}
by (j). So (6) is true.
\par\smallskip
$\bullet$ in order to prove property (7), we have to estimate
\begin{eqnarray*}
 ||u_{j(n)}^{(n)}-u_{n}^{(n)}||&=&\left|\left|
 \sum_{j=1}^{n}\frac{a_{j}}{\mu _{j(n)}-\lambda _{j}}e_{j}-\sum_{j=1}^{n}\frac{a_{j}}{\mu _{n}-\lambda _{j}}e_{j}
  \right|\right|\\
 &=&\left(\sum_{j=1}^{n-1}|a_{j}|^{2}
 \left|\frac{1}{\mu _{j(n)}-\lambda _{j}}-\frac{1}{\mu _{n}-\lambda _{j}}\right|^{2}\right)^{\frac{1}{2}}
 +|a_n| \left|\frac{1}{\mu_{j(n)}-\lambda_n}-\frac{1}{\mu_n-\lambda_n}\right|\\
&<&\varepsilon +|a_n| \, \frac{|\mu_n-\mu_{j(n)}|}{|\mu_{j(n)}-\lambda_n|\,.\,|\mu_n-\lambda_n|}\cdot
\end{eqnarray*}
by (i).
Now as $a_{n}=2^{-(n+1)}|\mu _{n}-\lambda _{n}|$, 
\begin{eqnarray*}
 |a_n| \frac{|\mu_n-\mu_{j(n)}|}{|\mu_{j(n)}-\lambda_n|\,.\,|\mu_n-\lambda_n|} &\leq&
2^{-(n+1)} \,\frac{|\mu_n-\lambda_n|+|\lambda_n-\mu_{j(n)}|}{|\mu_{j(n)}-\lambda_n|}\\
&\leq&
2^{-(n+1)} \, \left(1+\frac{|\mu_n-\lambda_n|}{|\mu_{j(n)}-\lambda_n|}\right)\\
&<&2^{-(n+1)} \,(1+2^{-n})
\end{eqnarray*}
by (j). It follows then from (a) that
$$||u_{j(n)}^{(n)}-u_{n}^{(n)}||\leq \varepsilon +2^{-(n+1)} \,(1+2^{-n})<2^{-n},$$
 so (7) is true.
\par\smallskip
$\bullet$ lastly, we have to estimate the quantities
$||T_{n}u_{i}^{(k)}-\mu _{i}u_{i}^{(k)}||$ for $k=1,\ldots,n-1$ and $i=1,\ldots,k$: since $T_{k}u_{i}^{(k)}=\mu _{i}u_{i}^{(k)}$, we have 
$$||T_{n}u_{i}^{(k)}-\mu _{i}u_{i}^{(k)}||=||\sum_{p=k+1}^{n}(T_{p}-T_{p-1})u_{i}^{(k)}||\leq
\sum_{p=k+1}^{n} ||(T_{p}-T_{p-1})u_{i}^{(k)}||.$$
Since $u_{i}^{(k)}$ belongs to $\textrm{sp}[e_{1},\ldots,e_{k}]$, we have $T_{p}u_{i}^{(k)}=D_{k}u_{i}^{(k)}+R_{p}u_{i}^{(k)}$ for $p\geq k$, so that 
\begin{eqnarray*}
(T_{p}-T_{p-1})u_{i}^{(k)}&=&(R_{p}-R_{p-1})u_{i}^{(k)}=
\pss{u_{i}^{(k)}}{b^{(p)}-b^{(p-1)}}a^{(p)}
\end{eqnarray*}
for $p\geq k+1$. Thus
\begin{eqnarray*}
 ||(T_{p}-T_{p-1})u_{i}^{(k)}||&\leq&||{b^{(p)}-b^{(p-1)}}||\,.\,|| u_{i}^{(k)}||\,.\,||a^{(p)}||.
\end{eqnarray*}
By the induction assumption and (5) which we have already proved for $p=n$, we know that (9') holds true for any $p\leq n$:
$||b^{(p)}-b^{(p-1)}||\leq 2^{-(p-1)}$ for $k+1\leq p \leq n$. Moreover for $k+1\leq p \leq n$, $||a^{(p)}||\leq 1$ by (3) which is true until step $n$, and so it remains to prove that $||u_i^{(k)}||\leq 3$ for any
$k=1,\ldots,n-1$. By the induction assumption and (6), we have
$||u_i^{(j)}-u_i^{(j-1)}||\leq 2^{-j}$ for $i+1\leq j\leq n-1$. Hence
$$||u_i^{(k)}-u_i^{(i)}||\leq \sum_{j=i+1}^k ||u_i^{(j)}-u_i^{(j-1)}||\leq \sum_{j=i+1}^k 2^{-j}\leq 2^{-i}$$ for any $1\leq k\leq n-1$. So $||u_i^{(k)}||\leq 2^{-i}+||u_i^{(i)}||$. Now for any $i\leq n-1$, we have by (7) of the induction assumption that $||u_{j(i)}^{(i)}-u_{i}^{(i)}||\leq 2^{-i}$ so that $||u_{i}^{(k)}||\leq 2.2^{-i}+||u_{j(i)}^{(i)}||$. Then since $i\leq n-1$ and $j(i)<i$ we can again estimate
$$||u_{j(i)}^{(i)}|| \leq ||u_{j(i)}^{(j(i))}||+2^{-j(i)}<2.2^{-j(i)}+||u_{j(j(i))}^{(j(i))}||.$$ Since $j(m)<m$ for every $m\geq 2$, there exists for each $i\leq n-1$ an integer $s_i$ such that $j^{[s_{i}-1]}(i)>j^{[s_{i}]}(i)$ and $j^{[s_{i}]}(i)=1$, where $j^{[s]}(i)$ denotes for each $s\geq 1$ the $s^{th}$ iterate of the function $j$. Thus 
$$||u_i^{(k)} ||\leq 2(2^{-i}+2^{-j(i)}+2^{-j(j(i))}+\ldots+2^{-j^{[s_{i}-1]}(i)}+2^{-1})+||u_1^{(1)}||\leq 3.$$
So  $||(T_{p}-T_{p-1})u_{i}^{(k)}||\leq 3\,.\, 2^{-(p-1)}$ for any 
$k+1\leq p\leq n$.
This yields that $$||T_{n}u_{i}^{(k)}-\mu _{i}u_{i}^{(k)}||<3\,\sum_{p=k+1}^{n}2^{-(p-1)}\leq 3\,.\,2^{-(k-1)}$$ and this estimate proves (8).

\end{document}